\documentclass[reqno]{amsart} 

\usepackage{amsmath}
\usepackage{amssymb}
\usepackage{amsthm}

\theoremstyle{plain}
\newtheorem{theo}{Theorem} 
\newtheorem{prop}[theo]{Proposition}

\newtheorem{coro}[theo]{Corollary}

\theoremstyle{definition}
\newtheorem{exam}{Example}
\newtheorem{defi}{Definition} 

\theoremstyle{remark}

\newcommand\Ob{\mathcal S(H)}
\newcommand\Prj{\mathcal P(H)}
\newcommand\Pf{\mathcal{PF}}
\newcommand\Pfun{\Pf(I,V)}
\newcommand\colo{\colon\,}
\newcommand\Meas{\mathcal M(\mathcal A)}
\newcommand\dom{\mathop{\mathrm{dom}}}
\newcommand\ran{\mathop{\mathrm{ran}}}

\newcommand\ovran{\mathop{\overline{\mathrm{ran}}}}
\newcommand\bck{\mathbin{\mbox{\raisebox{-.10ex}{$-$}%
\hspace{-1.21ex}\raisebox{.90ex}{.}}\,}}
\newcommand\supp{\mathop{\mathrm{supp}}}

\newcommand\RR{\mathbb R}
\newcommand\sminus\smallsetminus
\let\ph\phi
\renewcommand\phi\varphi
\newcommand\sqsub\sqsubset
\newcommand{\rcap}{\overleftarrow\cap}
\newcommand\rwedge{\overleftarrow\wedge}
\newcommand\rcurwedge{\overleftarrow\curlywedge}
\newcommand\para{\parallel}
\newcommand{\comp}{\mathrel{\mbox{\,\raisebox{.66ex}{$\scriptscriptstyle
        |$}\hspace{-.86ex}\raisebox{-.64ex}{$\circ$}}}}

\newcounter{perp}
\newenvironment{perpenum}{\begin{enumerate}
        \setlength{\itemindent}{1.94em}
        \setcounter{enumi}{\value{perp}}

        \setlength{\itemsep}{0pt}}
{\setcounter{perp}{\value{enumi}}\end{enumerate}}
\newcommand{\rperp}[1]{($\perp_{\ref{#1}}$)}

\newenvironment{myenum}{\begin{enumerate}

        \setlength{\itemsep}{0pt}}
        {\end{enumerate}}

\newcounter{vee}
\newenvironment{veeenum}{\begin{enumerate}
        \setlength{\itemindent}{1.94em}
        \setcounter{enumi}{\value{vee}}

        \setlength{\itemsep}{0pt}}
{\setcounter{vee}{\value{enumi}}\end{enumerate}}
\newcommand{\rvee}[1]{($\vee_{\ref{#1}}$)}

\newcounter{oplus}
\newenvironment{oplusenum}{\begin{enumerate}
        \setlength{\itemindent}{1.94em}
        \setcounter{enumi}{\value{oplus}}

        \setlength{\itemsep}{0pt}}
{\setcounter{oplus}{\value{enumi}}\end{enumerate}}
\newcommand{\roplus}[1]{($\oplus_{\ref{#1}}$)}

\newcounter{ominus}

\newcounter{minus}
\newenvironment{minusenum}{\begin{enumerate}
        \setlength{\itemindent}{1.94em}
        \setcounter{enumi}{\value{minus}}

        \setlength{\itemsep}{0pt}}
{\setcounter{minus}{\value{enumi}}\end{enumerate}}
\newcommand{\rminus}[1]{($-_{\ref{#1}}$)}

\newcounter{rwedge}

\newcounter{sqsub}
\newenvironment{sqsubenum}{\begin{enumerate}
        \setlength{\itemindent}{1.94em}
        \setcounter{enumi}{\value{sqsub}}

        \setlength{\itemsep}{0pt}}
{\setcounter{sqsub}{\value{enumi}}\end{enumerate}}
\newcommand{\rsqsub}[1]{($\sqsub_{\ref{#1}}$)}

\newcounter{rwed}

\begin{document}
\title[Further remarks on an order for quantum observables]%
{Further remarks on an order for quantum observables}
\author{J\=anis C{\=\i}rulis}
\date{\today}

\newcommand{\acr}{\newline\indent}
\address{
    Faculty of Computing\acr
    University of Latvia\acr
    Rai\c{n}a b., 19\acr
    Riga LV-1586\acr
    LATVIA}
\email{jc@lanet.lv}

\date{}

\thanks{This work was supported by ESF project
No.\ 2009/0216/1DP1.1.1.2.0/09/APIA/VIAA/044}

\subjclass{Primary 81Q10; Secondary 06A12, 46L10, 47L30, 81P10}
\keywords{bounded self-adjoint operator, generalized orthoalgebra, nearsemilattice, orthogonality, orthomodular lattice, quasi-orthomodular, quantum observable, skew meet}

\begin{abstract}
S.~Gudder and, later, S.~Pulmanov\'a and E.~Vincekov\'a, have studied in two recent papers a certain ordering of bounded self-adjoint operators on a Hilbert space. We present some further results on this ordering and show that some structure theorems of the ordered set of operators can be obtained in a more abstract setting of posets having the upper bound property and equipped with a certain orthogonality relation.
\end{abstract}

\maketitle

\section{Introduction}  \label{intro}

In \cite{ord1}, S.~Gudder introduced a certain order for quantum observables, in fact, on the set $\Ob$ of bounded self-adjoint operators on a complex Hilbert space $H$, and suggested to call it the \emph{logical order}. He demonstrated, in particular, that $\Ob$ is a generalized orthoalgebra; the logical order is by definition the natural order of this algebra. He also showed that this is not a lattice order, but nevertheless every pair of observables having a common upper bound has a join and a meet with respect to this order. Such a poset was called by him a \emph{near-lattice} (but see the beginning of the next section). Actually, every initial segment of $\Ob$ is even a $\sigma$-orthomodular lattice. Also the commutative case (in which observables are represented by random variables on a probability space) was considered in \cite{ord1}; as noted by the author of that paper, this case actually served as motivation and a source of intuition for results and proofs of the general case.

The properties of the new ordering were studied in more detail by S.~Pulmannov\'a and E.~Vincekov\'a in \cite{ord2}, where several results of \cite{ord1} were essentially improved. These authors observed that the logical order is actually a restriction of Drazin's order (nowdays usually called star order or *-order) introduced by him in  \cite{star0} for all bounded operators on $H$. They proved, for example, that $\Ob$  is even a weak generalized orthomodular poset. Moreover, this poset is \emph{bounded complete}, i.e., every subset bounded from above has a join and, correspondingly, every nonempty subset has a meet.

More recently, existence conditions of joins and meets in $\Ob$ (under the logical ordering), and representations of these operations have been discussed, e.g., in \cite{ord3,ord4,ord5,ord6,ord7}.

We present here some further results on the order structure of $S(H)$, and also fill two small gaps in proofs in \cite{ord1}. In particular, we obtain explicit descriptions for the Gudder join and meet operations in terms of operator composition and lattice operations on projectors, and a simple proof of bounded completeness of $\Ob$. On the other hand, we discuss some of the properties of the poset $\Ob$ in a more abstract setting, and show that the aforementioned properties of the logical order of observables are not quite independent.
For instance, any poset having the least element and possessing the so called \emph{upper bound property} (any pair of elements has a join if they have a common upper bound), if equipped with an appropriate orthogonality relation, carries a structure of a generalized orthoalgebra, in which every initial segment is an orthomodular lattice.
Moreover, there is a non-commutative total binary operation (called \emph{skew meet}) on $\Ob$ which, considered together with the partial join operation, turns the poset into a so called \emph{skew nearlattice}. This allows, on the one hand, to establish a link between structures arising in quantum logic and some branches of the theory of information systems (where the notion of skew nearlattice has emerged; see \cite{skew1}), and on the other hand, to apply to the algebra $\Ob$ certain general decomposition and structure theorems from \cite{skew2}. We do not address, however, these questions in the present paper.

The structure of the paper is as follows. The subsequent section contains the necessary background on posets having the upper bound property, known also as nearsemilattices in algebra and as a (simple version of) domains in database theory. Three principal examples of such posets, including $\Ob$, are considered in Section \ref{exam}. Some order properties of $\Ob$ are discussed also in Section \ref{more}. Section \ref{ortho} deals with so called quasi-orthomodular nearsemilattices, which mimic, in a sense, the generalized orthomodular lattices of \cite{jan}, and Section \ref{over}, with skew nearlattices.

\section{Preliminaries: nearsemilattices}  \label{perlim}

A \emph{nearlattice} is usually defined as a meet semilattice having the upper bound property. Therefore, every bounded complete poset is an example of a nearlattice. Equivalently, a nearlattice is a meet semilattice in which every initial segment is a lattice. Since early eighties, such structures have been intensively studied by W.~Cornish and his collaborators; see, e.g., \cite{nlatt1,nlatt2,nlatt3}. Some authors prefer order duals of such algebras \cite{ch}.

Arbitrary posets having the upper bound property were named (upper) \emph{near\-semi\-lattices} in \cite{nsemi1}. Thus, a nearlattice may be viewed also as a nearsemilattice that happens to be a meet semilattice.
Near-lattices mentioned in \cite{ord1} (see Introduction) is a weaker concept. 

We shall always assume that a near(semi)lattice has the least element $0$, and consider such structures as partial algebras of kind $(A,\vee,0)$, resp., $(A,\wedge,\vee,0)$, where $\wedge$ is, as usual, the meet operation and $\vee$ is the partial join operation. The following axiomatic description of nearsemilattices goes back to \cite[Section 1]{nsemi1} (for arbitrary near(semi)lattice terms $s$ and $t$, we write $s \comp t$ to mean that $s \vee t$ is defined under  that or other intended assignment of values to variables occurring in these terms).

\begin{prop}    \label{oper}
An algebra $(A,\vee,0)$, where $0$ is a nullary operation and $\vee$ is a partial binary operation, is a nearsemilattice if and only if it fulfils the conditions
\begin{veeenum}
\item $x \comp x$ and $x \vee x = x$,   \label{vidm}
\item if $x \comp y$, then $y \comp x$ and $x \vee y = y \vee x$, \label{vcom}
\item if $x \comp y$ and $x \vee y \comp z$, then $ y \comp z$, $x \comp y \vee z$ and $(x \vee y) \vee z = x \vee (y \vee z)$,\label{vass}
\item $x \comp 0$ and $x \vee 0 = x$.   \label{vnul}
\end{veeenum}
\end{prop}

The order relation on a nearsemilattice $A$ is recovered from the partial join operation as follows:
\begin{equation} \label{le/vee}
x \le y \text{ if and only if } x \comp y \text{ and } x \vee y = y,
\end{equation}
and then $x \vee y$ (when defined) is the join of $x$ and $y$ w.r.t.\ this order, while 0 is the least element.

\begin{coro}
An algebra $(A,\wedge,\vee,0)$ is a nearlattice if and only if $(A,\vee,0)$ is a nearsemilattice, $(A,\wedge)$ is a semilattice,  and the following absorption laws are fulfilled:
\begin{veeenum}
\item if $x \comp y$, then $x \wedge (x \vee y) = x$,   \label{wabs}
\item $x \wedge y \comp y$ and $(x \wedge y) \vee y = y$.    \label{vabs}
\end{veeenum}
\end{coro}

Observe that these laws can also be rewritten in a $\vee$-free form, where $\le$ is the ordering (\ref{le/vee}):
\begin{equation}    \label{vee:wedge}
\mbox{if $x \le y$, then $x \wedge y = x$, \ and \ $x \wedge y \le y$.}
\end{equation}

A nearsemilattice $A$ is said to be \emph{distributive} \cite{nsemi2}, if every initial segment of it is a distributive semilattice:
\begin{veeenum}
\item if $y \comp z$ and $x \le y \vee z$, then $x = y' \vee z'$ for some $y' \le y$ and $z' \le z$.     \label{nsdistr}
\end{veeenum}
In the case when $A$ happens to be a nearlattice, this reduces to the standard notion of a distributive nearlattice (every initial segment is a distributive lattice). 

A \emph{De Morgan complementation}, or just \emph{m-complementation}, on a poset is a unary operation $^-$ such that
\[
\mbox{$x^{--} = x$, \ and \ if $x \le y$, then $y^- \le x^-$} .
\]
By a \emph{sectionally m-complemented poset} we mean a poset with the least element, in which every initial segment $[0,p]$ is  equipped with an m-complementation $^-_p$ \cite[Section 2]{nsemi2}. In such a poset, a partial subtraction operation $\ominus$ may be defined  by
\[
x \ominus y = z \text{ if and only if } y \le x \text{ and $z = y^-_x$
in } [0,x].
\]
Lemma 2.5 of \cite{nsemi2}, together with Proposition 2.2, Theorem 2.3 and Corollary 2.4 of that paper, leads us to the following characterization of certain nearlattices.

\begin{prop} \label{subtr}
A sectionally m-complemented poset is a nearlattice if and only if
the partial subtraction $\ominus$ on it can be extended to a total operation $-$ satisfying conditions
\begin{minusenum}
\item if $x \le y$, then $z-y \le z-x$, \label{manti}
\item $x - (x - y) \le y$, \label{mdi}
\item $x-0=x$. \label{-0}
\end{minusenum}
\end{prop}

Namely, in an m-complemented nearlattice,
\begin{equation}    \label{-/o-}
x - y = x \ominus (x \wedge y)
 \end{equation}
 (see equation (7) in \cite{nsemi2}).
Explicit definitions of join and meet in 
terms of subtraction are not given in \cite{nsemi2}. We omit the tedious calculations and note without proof that
\[
x \wedge y = x - (x - y), \ \text{ and } \ x \vee y =
z -((z -x) \wedge (z - y)) \text{ whenever } x,y \le z .
\]An ordered algebra $(A,-,0)$, where $0$ is the least element and $-$ is a binary operation satisfying \rminus{manti}--\rminus{-0}, is a weak BCK-algebra in the sense of \cite{nsemi2}; see Proposition 2.2 therein.

\section{Examples}  \label{exam}

Let as now consider three examples of nearlattices. The first example may be considered as motivating the concept; the other two are borrowed from \cite{ord1,ord2}.

\begin{exam}[partial functions] \label{pfun}
Let $I$ and $V$ be nonempty sets, and let$\Pfun$ be the set of all partial functions from $I$ to $V$. It is partially ordered by set inclusion, and $(\Pfun, \cap,\cup,\lambda)$, where $\lambda$ is the nonwhere defined function, is a nearlattice with respect this order and usual operations $\cap$ and $\cup$. Actually, $\Pfun$ is even bounded complete, and total functions are just its maximal elements.

Observe that the union of two functions in $\Pfun$ is defined if and only if they agree on the common part of their domains. It is easily seen that $\dom(\phi \cup \psi) = \dom \phi \cup \dom \psi$ and $\dom(\phi \cap \psi) \subseteq \dom \phi \cap \dom \psi$.
Any nearlattice $(\Phi,\cup,\cap,\lambda)$, where $\Phi$ is a subset of some $\Pfun$, is called a \emph{functional nearlattice}. If the nearlattice contains all pairwise unions (of elements of $\Phi$) existing in $\Pfun$, it is said to be \emph{closed}.

Every initial segment of a functional nearlattice $\Phi$ is even a Boolean lattice. Then the corresponding BCK-subtraction (\ref{-/o-}) coincides in $\Phi$ with set subtraction.
\end{exam}

\begin{exam}[random variables]  \label{0fun}
Let $(\Omega,\mathcal A,\mu)$ be a probability space, and let $\Meas$ be the set of random variables on this space (i.e., measurable functions $\Omega \to \RR$). Put $\supp f := \{\omega \in \Omega\colo f(\omega) \neq 0\}$. For $f,g \in \Meas$, write $f \perp g $ if $fg = 0$ (i.e., $\supp f \cap \supp g = \varnothing$), and put $f \preceq g$ if there is a function $h \in \Meas$ such that $f \perp h$ and $f +h = g$.
The relation $\preceq$ is an order on $\Meas$. By \cite[Theorem 3.1]{ord1}, $f \preceq g$ iff $f(\omega) = g(\omega)$ whenever  $\omega \in \supp f$; equivalently, iff $f = g\chi_{\supp f}$ ($\chi_K$ stands for the characteristic function of a subset $K \subseteq \Omega$). Direct calculations show that it is a nearlattice ordering with the zero function $0$ as its least element (Theorem 3.5 in \cite{ord1}). The corresponding meet and join operations, $\curlywedge$ and $\curlyvee$, may be defined as follows:
\[
f \curlywedge g := 
g\chi_{\{\omega \in (\supp f \cap \supp g)\colo f(\omega) = g(\omega)\}} =
g\chi_{\{\supp f \cap \supp g \sminus \supp(f - g)\}}
\]
and, if $f,g \preceq h$,
\[
f \curlyvee g := h\chi_{(\supp f \cup \supp g)}.
\]
By the way, $f \curlywedge g = h\chi_{(\supp f \cap \supp g)}$ in this case. (We have changed the notation of \cite{ord1}, where these operations were denoted by $\wedge$ and $\vee$, respectively.)
Similarly, if $F$ is a subset of $\Meas$ with an upper bound $h$, then $h\chi_{(\bigcup\{\supp f\colo f \in F\})}$ is the least upper bound of $F\colo \Meas$ is bounded complete.
\end{exam}

Observe that this example is essentially subsumed under the previous one: the nearlattice $\Meas$ may be identified with the closed subalgebra $M_{\mathcal A}$ of $\mathcal{PF}(\Omega,\RR_0)$ (where $\RR_0 := \RR \sminus \{0\}$) obtained by replacing every function in $\Meas$ with its codomain restriction to $\RR_0$; this transfer is an isomorphism of $\Meas$ into $\mathcal{PF}(\Omega,\RR_0)$. In particular, every initial segment $[0,g]$ in $\Meas$ is a Boolean lattice, where the complementation of $f$ is $g - f$. Then the BCK-subtraction in $\Meas$ (which we denote by $\bck$) is given by
\[
g \bck f =  g - (f \curlywedge g) = g\chi_{\supp(f - g)}.
\]

\begin{exam}[quantum observables]
We return to the set $\Ob$ of bounded self-adjoint operators on a Hilbert space. The ensuing conventions follow to \cite{ord1}. Let $\Prj$ stand for the set of all projections (idempotent operators).  If $A \in \Ob$, denote by $\ran A$ the range of $A$ and by $\ovran A$, its closure. Denote the projection onto $\ovran A$ by $P_A$. For $A,B \in \Ob$, put $A \perp B$ if the composition $AB$ is the zero operator $O$ (equivalently, if $\ovran A \perp \ovran B$, or $P_AP_B = O$; see Lemma 4.1 in \cite{ord1}), and put $A \preceq B$ if $B = A + C$ for some $C \in \Ob$ with $C \perp A$. By \cite[Lemma 4.3]{ord1}, $A \preceq B$ if and only if $Ax = Bx$ whenever $x \in \ovran A$ (then $\ovran A \subseteq \ovran B$), or, equivalently, if $A = BP_A$ (then $B$ and $P_A$ commute). 
The relation $\preceq$ is the logical order on $\Ob$ mentioned in Introduction.

We shall say that the closed subspace $\ovran A$ of $H$ \emph{corresponds} to $A$. Recall that the transfer $P_A \mapsto \ovran A$ from projection operators to closed subspaces is one-to-one and onto; moreover,  $\Prj$ is ordered by
\[
\mbox{$P_1 \le P_2$ if and only if $P_1P_2 = P_1$ if and only if $P_1 = P_2P_1$},
 \]
 and then $P_A \le P_B$ if and only if $\ovran A \subseteq \ovran B$. In particular, $\Prj$ is lattice ordered, and $P_A \le P_B$ whenever $A \preceq B$. Corollary 4.4 in \cite{ord1} implies that $P_A \le P_B$ if and only if $P_A \preceq P_B$.

We denote the join and meet operations in the lattice $\Prj$ by $\vee$ and $\wedge$, and these operations in $\Ob$ (which may be partial), by $\curlyvee$ and $\curlywedge$, respectively (in \cite{ord1,ord2}, the same symbols $\vee$ and $\wedge$ are used for both purposes).
According to \cite[Corollary 4.13]{ord1}, meets and joins exist in $\Ob$ for every pair of elements bounded above (but see the beginning of the next section). Using this result as the base, it is shown in \cite[Corollary 4.7]{ord2} that $\Ob$ is even bounded complete, in particular, a nearlattice (evidently, $O$ is the least element in it).

It is easily seen that, in every interval $[O,B]$ of $\Ob$, $B - A$ is an m-complement of $A$. If the symbol $\bck$ stands for the BCK-subtraction (\ref{-/o-}) in the nearlattice $\Ob$, then (for arbitrary $A,B \in \Ob$)
\[
B\! \bck A = B - (A \curlywedge B) = B - BP_{A \curlywedge B} = B(I - P_{A \curlywedge B}).
\]
\end{exam}

This example is not fully subsumed under Example \ref{pfun}. The definition of $\preceq$ shows that $A = B$ if and only if
$A|\ovran A = B|\ovran B$. Therefore, an operator in $\Ob$ is completely determined by its restriction to the corresponding subspace of $H$. We thus can identify every operator $A$ with the partial function $A|\ovran A$ from $\Pf(H,H)$, and come in this way to a nearlattice of functions $S_H$ (ordered by set inclusion) isomorphic to $\Ob$. However, as the join of two subspaces generally differs from their set-theoretical union, this nearlattice need not be functional.

The set $S_H$ of partial operators admits also an immediate description. Theorem 4.12 of \cite{ord1} introduces, for any $B \in \Ob$, a subset
\[
L_B := \{P\colo P \le P_B \text{ and } BP = PB\}
\]
of $\Prj$. For example, if $A \preceq B$, then $P_A \in L_B$.
It is shown in the proof of the theorem that the mapping $\ph\colo A \mapsto P_A$ is an order isomorphism of the initial segment $[O,B]$ of $\Ob$ onto $L_B$.
Therefore,
\[ L_B = \{P\colo P = P_A \text{ and } A \preceq B \text{ for some } A \in \Ob\} = \{P_A\colo A \preceq B\}.
\]
It is now easily seen that $S_H = \{B|\ovran C\colo P_C \in L_B\}$. Indeed, $B|\ovran C$ belongs to $S_H$, by definition, if and only if  $B|\ovran C = A|\ovran A$ for some $A \in \Ob$, i.e. if and only if $\ovran C = \ovran A$ and $B|\ovran A = A|\ovran A$, i.e., if and only if $P_C = P_A$ and $A \preceq B$. 

It is also demonstrated in the proof of the mentioned theorem that
\begin{equation}    \label{prop}
\mbox{if $P \le P_B$, then $P = P_{BP}$ and, further, $BP \preceq B$} 
\end{equation}
whenever $BP \in \Ob$ i.e., whenever $BP = PB$. It follows that the inverse of the isomorphism $\ph$ takes a projection $P$ from $L_B$ into the operator $BP \in [O,B]$. Therefore, $[O,B] = \{BP\colo P \in L_B\}$.

\section{More on the ordering of $\Ob$}  \label{more}

In this section we discuss in more detail existence of joins and meets in the poset $\Ob$. 

It is noticed in \cite[Theorem 4.12]{ord1} that every poset $L_B$ is a lattice with respect to $\le$. According to this theorem, the initial segment $[O,B]$ of $\Ob$ is order isomorphic to $L_B$; thus, the segment itself is a lattice. Corollary 4.13 to this theorem then asserts that, in $\Ob$, every pair of elements bounded above has a meet and a join.
Of course, the meet of two elements in an initial segment of a poset is also their meet in the poset itself (and conversely).
However, this may be not the case with joins; this fact seems to be overlooked in \cite{ord1} when drawing the corollary. The subsequent theorem confirms that the corollary is nevertheless correct, and provides explicit descriptions of the corresponding partial operations (Corollary \ref{mj}).

Recall that the commutant of an operator $B$ (i.e., the set of all bounded operators commuting with $B$) is a von Neumann algebra. The lattice $L_B$, being an initial segment in the complete orthomodular lattice of all projections of this algebra, is therefore complete. This observation allows us to obtain a simple direct proof of existence of meets and joins in $\Ob$ for all bounded sets of operators.

\begin{theo}    \label{MJ}
Suppose that $\mathcal T \subseteq \Ob$, $\mathcal T \neq \varnothing$, and that $B \in \Ob$ is an upper bound of $\mathcal T\!$. Then 
\begin{myenum}
\item
$B(\bigwedge(P_A\colo A \in \mathcal T))$ is the greatest lower bound of $\mathcal T\!$, 
\item
$B(\bigvee(P_A\colo A \in \mathcal T))$ is the least upper bound of $\mathcal T\!$. 
\end{myenum}
\end{theo}
\begin{proof}  
\newcommand\mT{\mathcal T}
Assume that $A \preceq B$, i.e., $A = BP_A$, for all $A \in \mT\!$. Then all projections $P_A$ belong to the complete lattice $L_B$. This implies that the symbolic expressions in (a) and (b) present operators from $\Ob$.

(a) Let $P_\mT$ stand for the meet (in $\Prj$) of all projections $P_A$ with $A \in \mT$. For every $A \in \mT\!$, we have $BP_{\mT} = BP_AP_\mT = AP_\mT = AP_{BP_\mT}$ (by (\ref{prop}), as $P_\mT \le P_B$), i.e., $BP_\mT \preceq A$. Hence, $BP_\mT$ is a lower bound of $\mT\!$. Further, if $D$ is one more lower bound, then $D \preceq B$, $P_D \le P_\mT$ (as $P_D \le P_A$ for all $A \in \mT$), and $BP_\mT P_D = BP_D = D$. Thus, $D \preceq BP_\mT$. Therefore, $BP_\mT$ is the greatest lower bound of $\mT\!$.

(b) Let $P^\mT$ stand for the join (in $\Prj$) of all projections $P_A$ with $A \in \mT$. For every $A \in \mT\!$, we have $A = BP_A = BP^{\mT}P_A$, i.e., $A \preceq BP^\mT$. Hence, $BP^\mT$ is an upper bound of $\mT\!$; moreover, $BP^\mT \preceq B$, for $P^\mT$ belongs to the complete lattice $L_B$ containing all $P_A$. Further, if $D$ is one more upper bound, then likewise $DP_A = A$ for all $A \in \mT$ and $DP^\mT \preceq D$. 

Now, $(B-D)P_A = 0$ and, consequently, $P_{B-D}P_A = 0$ for every $A \in \mT$. Then in $\Prj$ also $P_A \le I - P_{B-D}$ for every $A$. Hence, $P^\mT \le I - P_{B-D}$, $P_{B-D}P^\mT = 0$ and, finally, $(B-D)P^\mT = 0$. Therefore, $BP^\mT = DP^\mT \preceq D$, and $BP^\mT$ is the least upper bound of $\mT$.
\end{proof}

Notice that $\sup \mathcal T = O$ if $\mathcal T = \varnothing$. Therefore, item (b) of the theorem immediately implies the result of \cite{ord2} that the  poset $\Ob$ is actually bounded complete (see Example 3 above), every initial segment of $\Ob$ is a complete sublattice of $\Ob$, and the embedding $P \mapsto BP$ of $L_B$ into $\Ob$  mentioned at the end of the previous section preserves all meets and joins.

As noted in the previous section, the ordering $\preceq$ agrees on $\Prj$ with the standard ordering $\le$ of projections. In \cite{ord1}, it is implicitly assumed without proof that then the meet (join) of two projection operators in $\Prj$ is also their meet (resp., join) in the more extensive poset $\Ob$. It follows from the theorem that this is indeed the case.

\begin{coro}  \label{pmj}
$\Prj$ is a complete sublattice of $\Ob$.
\end{coro}
\begin{proof}
Consider that $\mathcal T \subseteq \Prj$ (then $P_A = A$ for every $A \in \mathcal T$) and $B = I$.
\end{proof}

We write out the following significant particular case of the above theorem.

\begin{coro}    \label{mj}
If $A,B,C \in \Ob$ and $A,B \preceq C$, then
$A \curlywedge B$ and $A \curlyvee B$ exist, and
\[
A \curlywedge B = C(P_A \wedge P_B), \quad
A \curlyvee B = C(P_A \vee P_B).
\]
\end{coro}

By help of (\ref{prop}), we also conclude that
\begin{equation}    \label{hm}
P_{A \curlywedge B} = P_A \wedge P_B, \quad  P_{A \curlyvee B} = P_A \vee P_B
\end{equation}
whenever the pair $A,B$ is bounded.
Consequently, $\ovran(A \curlyvee B) = \ovran A \sqcup \ovran B$ and $\ovran(A \curlywedge B) = \ovran A \cap \ovran B$ in this case. Examples \ref{pfun} and \ref{0fun} suggest that generally $\ovran(A \curlywedge B) \subseteq \ovran A \cap \ovran B$.

The bounded completeness of $\Ob$ implies that item (a) of Theorem \ref{MJ} may be strengthened:
every nonempty subset of $\Ob$ has the greatest lower bound.
In particular, the set of all lower bounds of operators $A,B \in \Ob$ has the join, which is also its maximum element, i.e., $A \curlywedge B$. 
Observe that $\{P_C\colo C \preceq A \text{ and } C \preceq B\} = L_{A \curlywedge B}$, so that $P_{A \curlywedge B}$ belongs to this set and is the greatest element in it. 
Therefore, we come to the following description of meet of observables in a general case.

\begin{coro}
For every pair of observables $A$ and $B$, their meet exists and
\[ 
A \curlywedge B = B(\max\{P_C\colo C \preceq A \text{ and } C \preceq B\}) .
\] 
\end{coro}

This particular result admits also a simple direct proof (cf.\ Proposition 3.2 in \cite{star}).
Let us consider the following subset of $\Prj$:
\[
L_{A,B} := \{P\colo P \le P_A, P \le P_B,  AP = PA, BP = PB \text{ and } P \perp P_{A - B}\}.
\]
If $C \preceq A,B$, then $AP_C = C = BP_C$, $(A-B)P_C = O$ and $P_C \perp P_{A - B}$. 
Moreover, then also $P_C \in L_A,L_B$. Therefore, $P_C \in L_{A,B}$. On the other hand, if $P  \in L_{A,B}$, then 
$P \perp P_{A - B}$ and $P \perp (A - B)$, i.e., $(A - B)P = O$ and $AP = BP =: C$. 
 As also $P \in L_A$ and $P \in L_B$, we conclude by (\ref{prop}) that $C \preceq A,B$. Therefore,
\(
L_{A,B} = \{P\colo P = P_C \text{ and } C \preceq A,B \text{ for some } C \in \Ob\}
\)
and, finally,
\[
L_{A,B} = \{P_C\colo C \preceq A \text{ and } C \preceq B\} \subseteq L_A \cap L_B .
\]

The definition of $L_{A,B}$ can be rewritten in the form
\[
L_{A,B} = \{P \in \{A,B\}'\colo P \le P_A \wedge P_B \wedge (I - P_{A - B})\},
\]
where $\{A,B\}'$ is the commutant of $\{A,B\}$ and $I$ is the identity operator. Therefore, $L_{A,B}$ is a bounded subset of the complete lattice of all projections of the von Neumann algebra $\{A,B\}'$ and, hence, has the join $P^*$. Evidently, $P^*$ is even the maximum element of $L_{A,B}$. It is easily seen that then 
the operator $BP^*$ is the g.l.b.\ of $A$ and $B$ in $\Ob$.
Indeed, since $P^* \in L_{A,B}$, it follows that $P^* \in L_A, L_B$ and $P^* = P_C$ for some $C$ with $C \preceq A,B$. So, $C = AP^* = BP^*$. Suppose that $D$ is one more lower bound of $A$ and $B$. As $P_D \in L_{A,B}$, it follows that $P_D \le P^* = P_C$. On the other hand, $P_D \in L_A,L_B$. Recall that the transfer $\ph\colo  E \mapsto P_E$ (Section \ref{exam}) is an order isomorphism of $[O,B]$ onto $L_B$; therefore $D \preceq C$, and $C = A \curlywedge B$.

Of course, $AP^*$ also is the g.l.b.\ of $A$ and $B$.
We already know that actually $P^* = P_{A \curlywedge B}$.

\section{Quasi-orthomodular nearsemilattices}   \label{ortho}

In \cite{jan}, M.F.\ Janowitz introduced the notion of a generalized orthomodular lattice, which was defined as a lattice equipped with an appropriate orthogonality relation. We take up the idea and consider in this section nearsemilattices with orthogonality.

A binary relation $\perp$ on a poset $A$ with the least element $0$ is said to be an \emph{orthogonality}, if  it satisfies the conditions
\begin{perpenum}
\item if $x \perp y$, then $y \perp x$,  \label{psym}
\item if $x \le y$ and $y \perp z$, then $x \perp z$,  \label{panti}
\item $x \perp 0$.  \label{pnul}
\end{perpenum}

For example, if $^-$ is an m-complementation on $A$ and a relation $\perp$ is defined by $x \perp y$ iff $y \le x^-$, then $\perp$ is an orthogonality on $A$. We say that it is \emph{induced} by the m-complementation $^-$. Evidently, the induced orthogonality is \emph{additive} in the sense that
\begin{perpenum}
\item if $x \perp y$, $x \perp z$ and $y \comp z$, then $x \perp y \vee z$.  \label{paddit}
\end{perpenum}

\begin{defi}
Suppose that $(A,\vee,0)$ is a nearsemilattice and that $\perp$ is an orthogonality on it. The algebraic system $(A, \vee, \perp, 0)$ is called a \emph{quasi-orthomodular nearsemilattice}
if the following additional conditions are fulfilled:
\begin{perpenum}
\item if $x \perp y$, then $x \comp y$, \label{perpcomp}
\item if $x \le y$, then $y = x \vee z$ for some $z$ with $x \perp z$, \label{vdiff}
\item if $x \perp y$, $x \perp z$ and $y \le x \vee z$, then $y \le z$.    \label{vcanc<}
\end{perpenum}
\end{defi}

In a quasi-orthomodular nearsemilattice, the following cancellation law holds:
\begin{perpenum}
\item if $x \perp y$, $x \perp z$ and $x \vee y = x \vee z$, then $y = z$.    \label{vcanc}
\end{perpenum}
In particular, 
\begin{perpenum}
\item if $x \perp x$, then $x = 0$.     \label{pdeg}
\end{perpenum}
Also, \rperp{paddit} together with \rperp{perpcomp} implies that
\begin{perpenum}
\item every finite set of mutually orthogonal elements has a join
\end{perpenum}
(recall that in a nearsemilattice $x \comp y$ iff $x$ and $y$ have a common upper bound).

\setcounter{exam}{0}
\begin{exam}[continuation]
In a functional nearsemilattice, put $\phi \perp \psi$ if $\dom\phi \cap \dom\psi = \varnothing$. Then $(\Pfun, \cup, \perp,\lambda)$ is a quasi-ortho\-modular nearsemilattice with an additive orthogonality.
The existential condition \rperp{vdiff} may fail in an arbitrary closed functional nearsemilattice. The other existential condition \rperp{perpcomp} is ensured by closedness and may fail in an arbitrary functional nearsemilattice.
\end{exam}

\begin{exam}[continuation]
The relation $\perp$ on $\Meas$ introduced in the Section \ref{exam} is an orthogonality, and $(\Meas, \curlyvee, \perp,0)$ is a quasi-orthomodular nearsemilattice.  The orthogonality satisfies also \rperp{paddit} (Theorem 3.2(b) in \cite{ord2}).
Taking into account that \rperp{vdiff} follows directly from the definition of $\preceq$, these facts can be derived from the previous example, because the isomorphism of $\Meas$ onto the closed functional nearsemilattice $M(A)$, which was described in Section \ref{exam}, preserves orthogonality.
\end{exam}

\begin{exam}[continuation]
Likewise, $(\Ob,\curlyvee,\perp,O)$ is a quasi-orthomodular nearsemilattice for $\perp$ defined as in Section \ref{exam}.
Indeed, (i) if $AB=O$, then $BA = O$ \cite[Lemma 4.1]{ord1}; (ii) if $A \preceq B$ and $BC = O$, then $A = BP_A = P_AB$ and $AC = O$; (iii) $AO = O$; thus, $\perp$ is an orthogonality. Further, (v) if $AB = O$, then $A \preceq A + B$ and $B \preceq A + B$ by the definition of $\preceq$, and $A \comp B$ (as $\Ob$ is a nearsemilattice)
[in addition, (v') $A + B = A \curlyvee B$; see the proof of Theorem 4.3 in \cite{ord2}];
(vi) if $A \preceq B$, then, by the definition and (v'), $B = A + C = A \curlyvee C$ for some $C$ with $A \perp C$; 
(vii) if $AB = O$, $AC = O$ and $B \preceq A \curlyvee C$, then  $A + C = B + D$ for some $D$ with $DB = O$,  $(D - A)B = O$ and, further, $C = B + (D - A)$, i.e., $B \preceq C$.

It was proved in Theorem 4.3 of \cite{ord2} that a weakened version of \rperp{paddit} holds in $\Ob$. The orthogonality on $\Ob$ is in fact additive.
Indeed, suppose that $D \perp A$, $D \perp B$ and $A \comp B$.
Then also $P_D \perp P_A$ and $P_D \perp P_B$. As the lattice $\Prj$ is orthomodular, it follows that $P_D \perp P_A \vee P_B$.
On the other hand, $P_A \vee P_B = P_{A \curlyvee B}$---see (\ref{hm}). Therefore, $P_D \perp P_{A \curlyvee B}$ and
$D \perp A \curlyvee B$. 

The isomorphism of $\Ob$ onto  the nearsemilattice
 $S_H$ described in the Section 3 preserves orthogonality; so we may conclude that $S_H$ is quasi-orthomodular.
\end{exam}

Theorem 4.12 of \cite{ord1} asserts that every initial segment of $\Ob$ 
is isomorphic to an orthomodular lattice. We are now going to generalize this structure theorem to arbitrary quasi-orthomodular nearsemilattices.
Recall that an orthocomplementation, or an \emph{o-complementation}, for short, on a bounded poset $(P,\le,0,1)$ is an m-complementation $^-$ such that $1 = x \vee x^-$ (equivalently,  $0 = x \wedge x^-$) for all $x \in P$. Observe that then the induced orthogonality satisfies \rperp{pdeg}. An o-complemented poset is \emph{orthomodular} if this orthogonality satisfies also \rperp{perpcomp} and \rperp{vdiff}; conditions \rperp{vcanc<} and \rperp{paddit} are fulfilled in every such a poset. An \emph{orthomodular lattice} is a lattice-ordered orthomodular poset.

\begin{theo} \label{oml}
Every initial segment of a quasi-orthomodular nearsemilattice $A$ is an orthomodular lattice, where joins and meets agree with those existing in  $A$.
\end{theo}
\begin{proof}
Due to \rperp{vdiff} and \rperp{vcanc}, there is, for every  $x \le p$, a unique element $y \in [0,p]$ such that $x \perp y$ and $x \vee y = p$; we denote this element by $x^-_p$.

The mapping $x \mapsto x^-_p$ is an m-complementation on $[0,p]$. By the definition,  $(x^-_p)^-_p \perp x^-_p$ and $(x^-_p)^-_p \vee x^-_p = p$; hence $(x^-_p)^-_p = x$. If $x \le y \le p$, then $x \perp y^-_p$ by \rperp{panti} and $y^-_p \le x^-_p$ by \rminus{vcanc<} (observe that $y^-_p \le x \vee x^-_p$).

As $x \vee x^-_p = p$, the m-complementation $^-_p$ is even an o-complementation in $[0,p]$, and the interval $[0,p]$ is an o-complemented poset. In virtue of \rperp{perpcomp} and \rperp{vdiff}, the poset is orthomodular. Since an initial segment $[0,p]$ of a nearsemilattice is actually an upper semilattice by definition, the o-complemetantion $^-_p$ turns it into a lattice. The final assertion is now trivial.
\end{proof}

The following conclusion is immediate (recall that a distributive semilattice that happens to be a lattice is also distributive as a lattice).

\begin{coro}
In a distributive quasi-orthomodular nearsemilattice, every initial segment is a Boolean algebra.
\end{coro}

It is observed in \cite{ord1} that $\Meas$ is a generalized orthoalgebra, and \cite[Theorem 4.2]{ord1} asserts that so is also $\Ob$. This result extends to arbitrary quasi-orthomodular nearsemilattices.

\begin{defi}    \label{def:genortho}
A \emph{generalized orthoalgebra} is a system $(A,\oplus,0)$, where $\oplus$ is a partial binary operation and $0$ is a nullary operation on $A$, satisfying the conditions (we write here, for arbitrary terms $s$ and $t$, $s \perp t$ to mean that $s \oplus t$ is defined):
\begin{oplusenum}
\item if $x \perp y$, then $y \perp x$ and $x \oplus y = y \oplus x$,    \label{ocom}
\item if $x \perp y$ and $x \oplus y \perp z$, then \\
    \hspace*{0mm}\hfill $y \perp z, x \perp y \oplus z$ and $(x \oplus y) \oplus z = x \oplus (y \oplus z)$,  \label{oass}
\item $x \perp 0$ and $x \oplus 0 = x$,  \label{onul}
\item if $x \perp y$, $x \perp z$ and $x \oplus y = x \oplus z$, then $y = z$,       \label{ocanc}
\item if $x \perp x$, then $x = 0$.  \label{odeg}
\end{oplusenum}
The relation $\le$ defined by
\begin{equation}  \label{le/oplus}
\mbox{$x \le y$ if and only if $y = x \oplus z$ for some $z$ with $x \perp z$},
\end{equation}
is an order on a genralized orthoalgebra $A$ and is called its \emph{natural ordering}.
\end{defi}

\begin{theo}    \label{osumperp}
Suppose that $(A,\vee,\perp,0)$ is a quasi-orthomodular nearsemilattice and that $\oplus$ is a binary operation on $A$ defined by
\begin{equation} \label{oplus/vee}
x \oplus y = z \text{ if and only if } x \perp y \text{ and } z = x \vee y.
\end{equation} 
Then $(A,\oplus,0)$ is a generalized orthoalgebra, and its natural ordering coincides with the nearsemilattice ordering of $A$.
\end{theo}
\begin{proof}
In virtue of \rperp{perpcomp}, the axioms \roplus{ocom}--\roplus{odeg} are easy consequences of \rvee{vcom}+\rperp{psym}, \rvee{vass}+\rperp{panti}, \rvee{vnul}+\rperp{pnul}, \rperp{vcanc} and \rperp{pdeg}, respectively. We only note in connection with \roplus{oass} that the supposition $x \oplus y \perp z$ implies that (i) $x,y \perp z$ (see \rperp{panti}) and (ii) $x \vee y \comp z$, i.e., $x,y,z \le p$ for some $p$. Then  (in the orthomodular lattice $[0,p]$) $y,z, y \vee z \le x^-_p$ and $x \perp y \vee z$, i.e., $x \perp y \oplus z$. The equivalence (\ref{le/oplus}) follows from \rperp{vdiff}.
\end{proof}

\setcounter{exam}{1}
\begin{exam}[continuation]
The operation $\oplus$ defined by $f \oplus g := f + g$ for $f \perp g$ turns $\Meas$ into a generalized orthoalgebra \cite{ord1}. See the item (i) in the proof of \cite[Theorem 3.2]{ord2} for (\ref{oplus/vee}).
\end{exam}
\begin{exam}[continuation]
If $A \perp B$, then put $A \oplus B := A + B$ \cite{ord1}. Then $\Ob$ is a generalized orthoalgebra \cite[Theorem 4.2]{ord1}. Evidently, $A \oplus B$ is an upper bound of $A$ and $B$, and Corollary 4.5 in \cite{ord1} says that it is actually a least upper bound. Thus, (\ref{oplus/vee}) also holds in $\Ob$.
\end{exam}

The theorem implies that an orthomodular nearsemilattice satisfying \rperp{paddit} is a generalized orthomodular poset (see \cite{ord2} for an appropriate version of the latter notion).

We say that a quasi-orthomodular nearsemilattice has the \emph{Riesz decomposition property} if it satisfies the condition
\begin{oplusenum}
\item if $y \perp z$ and $x \le y \oplus z$, then $x = y' \oplus z'$ for some $y' \le y$ and $y' \le z$,
\end{oplusenum}
where $\oplus$ is the operation (\ref{oplus/vee}) (cf.\ \cite{newtrends} or \cite[Sect.\ 2]{ord2}). By \rperp{perpcomp} and (\ref{oplus/vee}), this property turns out to be equivalent to distributivity.

\begin{theo}
A quasi-orthomodular nearsemilattice is distributive if and only if it has the Riesz decomposition property.
\end{theo}
\begin{proof}
In a quasi-orthomodular nearsemilattice $A$, if $y \perp z$ and $x \le y \oplus z$, then $x \le y \vee z$ and, by distributivity, $x = y' \vee z'$ for some $y' \le y$ and $z' \le z$. But $y' \perp z'$ in virtue of \rperp{panti}; therefore, $x = y' \oplus z'$. Now suppose that $A$ has the Riesz decomposition property and that $y \comp z$ and $x \le y \vee z$. By \rperp{vdiff}, $y \vee z = y \vee z_0$ for some $z_0$ with $z_0 \perp y$. As $z \le y \vee z_0$, there are $y_1 \le y$ and $z_1 \le z_0$ such that $y_1 \perp z_1$ and $z = y_1 \vee z_1$. Then $y \vee z = y \vee z_1$, $z_1 \le z$ and, in virtue of \rperp{panti}, $y \perp z_1$. It follows that $x = y' \vee z'$ for some $y' \le y$ and $z' \le z_1 \le z$.
\end{proof}

\section{Skew meets on quasi-orthomodular nearsemilattices}  \label{over}

We concentrate in this section on quasi-orthomodular nearsemilattices that admit a non-commutative meet-like operation, and demonstrate that the nearsemilattices in   Examples 1--3 belong to this type of nearsemilattices.

Let $A$ be a quasi-orthomodular nearsemilattice, and assume that orthogonality in it is additive. For $x, y \in A$, we write $x  \sqsub y$ to mean that, for every $z \in A$, $z \perp y$ only if $z \perp x$. We say that $x$ is \emph{overridden} by $y$, if $x \sqsub y$.
The overriding relation $\sqsub$ has the following properties:
\begin{sqsubenum}
\item $\sqsub$ is reflexive and transitive, \label{sqpre}
\item if $x \le y$, then $x \sqsub y$.
\item if $x \sqsub y$ and $x \comp y$, then $x \le y$, \label{sqle}
\item if $x \sqsub z$, $y \sqsub z$ and $x \comp y$, then $x \vee y \sqsub z$. \label{sqadd}
\end{sqsubenum}
Only  \rsqsub{sqle} 
requires some comment. By \rperp{vdiff}, $x \vee y = x_0 \vee y$ for some $x_0$ with $x_0 \perp y$. Since $x \sqsub y$, then $x_0 \perp x$ and, further, $x_0 \perp x \vee y$ (see \rperp{paddit}). Now, $x_0 = 0$ by \rperp{panti} and \rperp{pdeg}.

Therefore, $\sqsub$ is a preorder. We denote by $\para$ the equivalence relation on $A$ induced by it:  $\sqsub$: $x \para y$ iff $x \sqsub y$ and $y \sqsub x$
Abstract overriding relations satisfying \rsqsub{sqpre}--\rsqsub{sqadd} and one more condition
\begin{sqsubenum}
\item if $x \sqsub y$, then $x \para x' \le y$ for some $x'$, \label{sqproj}
\end{sqsubenum}
were introduced in \cite[Definition 2.2]{ov}. Observe that the element $x'$ in the right side of \rsqsub{sqproj} is uniquely defined; in fact, $x' = \max\{u\colo u \sqsubset x \text{ and } u \le y\}$. Indeed, $x' \sqsub x$ and $x' \le y$ by the definition of $\para$; on the other hand, if $u \sqsub x$ and $u \le y$, then $u \sqsub x'$ and $u \comp x'$, whence $u \le x'$ by \rsqsub{sqle}.

If the element
\[
x \rwedge y := \max\{u\colo u \sqsub x \text{ and } u \le y\}
 \]
exists for some $x$ and $y$, we call it a (right-handed) \emph{skew meet} of $x$ and $y$.  

\setcounter{exam}{0}
\begin{exam}[continuation]
In $\Pfun$, $\phi \sqsub \psi$ if and only if $\dom\phi \subseteq \dom \psi$. This overriding relation on $\Pfun$ satisfies \rsqsub{sqproj}: if $\phi \sqsub \psi$ and $\phi' := \psi|\dom \phi$, then $\phi \para \phi' \subseteq \psi$. However, \rsqsub{sqproj} may fail in an arbitrary functional nearsemilattice (even if the latter is  closed). More generally, the skew intersection $\rcap$ on $\Pfun$ is totally defined and is given by $\phi \rcap \psi = \psi|\dom(\phi \cap \psi)$ (see \cite{skew1}). Observe that $\phi \subseteq \psi$ iff $\phi \rcap \psi = \phi$, $\phi \perp \psi$ iff $\phi \rcap \psi = \lambda$ and $\phi \sqsub \psi $ iff $\psi \rcap \phi = \phi$.
\end{exam}

\begin{exam}[continuation]
In $\Meas$, $f \sqsub g$ if and only if $\supp f \subseteq \supp g$.  In connection with \rsqsub{sqproj}, observe that if $f \sqsub g$, then the function $f'$ which agrees with $g$ on $\supp f$ and vanishes outside $\supp f$ belongs to $\Meas$. More generally, the skew meet $\rcurwedge$ on $\Meas$ is a total operation and is given by $f \rcurwedge g = g\chi_{(\supp f \cap \supp g)}$.
\end{exam}

\begin{exam}[continuation]
In $\Ob$, $A \sqsubset B$ if and only if $\ovran A \subseteq \ovran B$. Indeed, suppose that $C \perp A$ whenever $C \perp B$, and choose  $x \in \ovran A$.
Take for $C$ the projection onto the closed subspace $[x]$ spanned by $x$. Then $C \not\perp A$ and, consequently, $C \not\perp B$. Hence, $[x] \subseteq \ovran B$, and $x \in \ovran B$. Conversely, if $\ovran A \subseteq \ovran B$ and $C \perp B$, then $\ovran C \cap \ovran A \subseteq \ovran C \cap \ovran B $ and $C \perp A$.

Equivalently, $A \sqsub B$ if and only if $P_A \le P_B$. The axiom \rsqsub{sqproj} means that if $P_A \le P_B$, then $P_A = P_{A'}$ and $A' \preceq B$ for some $A' \in \Ob$ (equivalently,  $P_A \in L_B$, or $B|\ovran A \in S_H$). A natural candidate for $A'$ is the operator $BP_A$; we, however cannot prove that it is self-adjoint, i.e., that $B$ and $P_A$ commute.

However,  the skew meet operation $\rcurwedge$ in $\Ob$ is total. The bounded subset $L^A_B := \{P\colo P \le P_A, P_A \le P_B \text{ and } BP = PB\}$ of the complete lattice $L_B$ always has the greatest element $P^*$. On the other hand,
 \[L^A_B = \{P_C \in L_B\colo  P_C \le P_A\} =\{P_C\colo  C \sqsub A \text{ and } C \preceq B\}.
\]
Now, a reasoning similar to that at the end of Section \ref{more} demonstrates that $BP^*$ is the skew meet of $A$ and $B$.
\end{exam}

We end with a theorem describing characteristic properties of the skew meet operation $\rwedge$.

\begin{theo}
Suppose that $A$ is a quasi-orthomodular nearsemilattice with an additive orthogonality and that all skew meets $x \rwedge y$ in $A$ exist. Then the algebra $(A, \rwedge)$ is an idempotent semigroup, and the following conditions 
are fulfilled:
\begin{gather}    \label{rwedge1}
\mbox{$x \rwedge y \le y$, \quad $x \rwedge y \sqsub x$}, \\
\mbox{$x \le y$ iff $x \rwedge y = x$,  \quad  $y \sqsub x$ iff  $x \rwedge y = y$}.   \label{rwedge2}
\end{gather}
Moreover, $\rwedge$ is commutative in every initial segment of $A$.
\end{theo}
\begin{proof}
Evidently, the operation $\rwedge$ is idempotent, and both inequalities (\ref{rwedge1}) are trivial. To prove that the operation is associative, observe that
\begin{multline} \label{left}
(x \rwedge y) \rwedge z = \max(v\colo v \sqsub \max(u\colo u \sqsub x \text{ and } u \le y) \text{ and } v \le z) \\
= \max(v\colo v \sqsub u \sqsubset x \text{ and } u \le y \text{ for some } u, \text{ and } v \le z)
\end{multline}
and
\begin{multline}  \label{right}
x \rwedge (y \rwedge z) = \max(v\colo v \sqsubset x \text{ and } v \le \max(u\colo u \sqsubset y \text{ and } u \le z)) \\
= \max(v\colo v \sqsubset x, \text{ and } v \le u \le z \text{ and } u \sqsubset y \text{ for some } u) \\
= \max(v\colo v \sqsub x, v \le z \text{ and } v \sqsub y).
\end{multline}
Now notice that if an element $v$ satisfies, for some $u$, the conditions
\( v \sqsub u \sqsubset x,\ u \le y \text{ and } v \le z \)
from (\ref{left}), then it satisfies
also the conditions 
\( v \sqsubset x, v \le z \text{ and } v \sqsubset y\)
from (\ref{right}). Conversely, if the latter triple of conditions is satisfied, then also the former one is satisfied with $u = x \rwedge y$. We only note that if $v \sqsub x,y$, then there is $v'$ such that $v \para v' \le y$ (see \rsqsub{sqproj}), and, further, $v' \le u$; it follows that $v \sqsub u$.

Therefore, the maxima in (\ref{left}) and (\ref{right}) are equal.
Next, if $x \le y$ and $u \le y$, then $x \comp u$ and, if also $u \sqsub x$, then $u \le x$. Thus, $x \rwedge y = x$. The converse is evident by (\ref{rwedge1}), and this proves the first identity in (\ref{rwedge2}). 
Further, if $y \sqsubset x$, then $y \le x \rwedge y$, and the converse again comes from (\ref{rwedge1}); this proves the other identity. 

At last, it follows from \rsqsub{sqle} by virtue of \rperp{panti} that $x \rwedge y = x \wedge y = y \rwedge x$ when $x \comp y$.
\end{proof}

The theorem shows that the algebra $(A,\vee,\rwedge,0)$ is a right normal skew nearlattice in the sense of \cite{skew1,skew2}. Thus, in particular, $\Ob$ is such an algebra.

\end{document}